\documentclass[12pt]{amsart}
\usepackage[margin=1.0in]{geometry}
\usepackage{amssymb}
\usepackage{dsfont}
\usepackage[T1]{fontenc}
\usepackage[utf8]{inputenc}
\usepackage{fourier}
\usepackage[english]{babel}
\usepackage[dvipsnames,usenames]{color}

\newcommand{\Nn}{\mathbb{N}}
\newcommand{\D}{\mathbb{D}}

\newcommand{\T}{\mathbb{T}}

\newcommand{\B}{\mathcal{B}}

\newcommand{\Ho}{{\mathcal{H}}^1}
\newcommand{\Ht}{{\mathcal{H}}^2}

\newcommand{\Hp}{{\mathcal{H}}^p}

\newcommand{\Real}{\operatorname{Re}}
\newcommand{\Imag}{\operatorname{Im}}

\theoremstyle{plain}
\newtheorem{theorem}{Theorem}
\newtheorem{lemma}{Lemma}
\newtheorem{corollary}{Corollary}
\theoremstyle{remark}
\newtheorem*{question}{Question}
\newtheorem*{remark}{Remark}

\begin{document}
\title[The multiplicative Hilbert matrix]{The multiplicative Hilbert matrix}

\author[O. F. Brevig]{Ole Fredrik Brevig}
\address{Ole Fredrik Brevig\\Department of Mathematical Sciences \\ Norwegian University of Science and Technology \\
NO-7491 Trondheim \\ Norway}
\email{ole.brevig@math.ntnu.no}

\author[K.-M. Perfekt]{Karl-Mikael Perfekt}
\address{Karl-Mikael Perfekt\\Department of Mathematical Sciences \\ Norwegian University of Science and Technology \\
 NO-7491 Trondheim \\ Norway}
\email{karl-mikael.perfekt@math.ntnu.no}

 \author[K. Seip]{Kristian Seip}
\address{Kristian Seip\\Department of Mathematical Sciences \\ Norwegian University of Science and Technology \\
NO-7491 Trondheim \\ Norway}
\email{seip@math.ntnu.no}

\author[A. Siskakis]{Aristomenis G. Siskakis}
\address{Aristomenis G. Siskakis \\ Department of Mathematics \\ Aristotle University of Thessaloniki \\
54124, Thessaloniki \\ Greece}
\email{siskakis@math.auth.gr}

\author[D. Vukoti\'{c}]{Dragan Vukoti\'{c}}
\address{Dragan Vukoti\'{c} \\ Departamento de Matem\'{a}ticas \\ M\'{o}dulo 17 \\ Facultad de Ciencias \\
Universidad Aut\'{o}noma de Madrid \\ 28049 Madrid
\\ Spain}
\email{dragan.vukotic@uam.es}

\thanks{The research of the first and the third author is supported by Grant 227768 of the Research Council of Norway.
The fifth author's work was supported by MTM2015-65792-P  by MINECO, Spain and ERDF (FEDER). This research was initiated while the third author served as a one month visitor in the ``Postgraduate Excellence
Program'' at the Department of Mathematics at Universidad Aut\'{o}noma de Madrid.}


\begin{abstract}
	It is observed that the infinite matrix with entries $(\sqrt{mn}\log (mn))^{-1}$ for $m, n\ge 2$ appears as the matrix
of the integral operator $\mathbf{H}f(s):=\int_{1/2}^{+\infty}f(w)(\zeta(w+s)-1)dw$ with respect to the basis $(n^{-s})_{n\ge 2}$;
here $\zeta(s)$ is the Riemann zeta function and $\mathbf{H}$ is defined on the Hilbert space ${\mathcal H}^2_0$ of Dirichlet
series vanishing at $+\infty$ and with square-summable coefficients. This infinite matrix defines a multiplicative Hankel
operator according to Helson's terminology or, alternatively, it can be viewed as a bona fide (small) Hankel operator on the
infinite-dimensional torus $\mathbb{T}^{\infty}$. By analogy with the standard integral representation of the classical Hilbert matrix,
this matrix is referred to as the multiplicative Hilbert matrix. It is shown that its norm equals $\pi$ and that it has a purely
continuous spectrum which is the interval $[0,\pi]$; these results are in agreement with known facts about the classical Hilbert
matrix. It is shown that the matrix $(m^{1/p} n^{(p-1)/p}\log (mn))^{-1}$ has norm $\pi/\sin(\pi /p)$ when acting
on $\ell^p$ for $1<p<\infty$. However, the multiplicative Hilbert matrix fails to define a bounded operator
on ${\mathcal H}^p_0$ for $p\neq 2$, where ${\mathcal H}^p_0$ are $H^p$ spaces of Dirichlet series. It remains an
interesting problem to decide whether the analytic symbol $\sum_{n\ge 2} (\log n)^{-1} n^{-s-1/2}$ of the multiplicative
Hilbert matrix arises as the Riesz projection of a bounded function on the infinite-dimensional torus $\mathbb{T}^\infty$.
\end{abstract}

\maketitle

\section{Introduction}
The classical Hilbert matrix
\[A:=\left(\frac{1}{m+n+1}\right)_{m,n\ge 0}\]
is the prime example of an infinite Hankel matrix, i.e., a matrix whose entries $a_{m,n}$ only depend on the sum $m+n$. The Hilbert matrix can be viewed as the matrix of the integral operator
\begin{equation} \label{classical}
	\mathbf{H}_a f(z):=\int_{0}^1 f(t) (1-zt)^{-1} dt
\end{equation}
with respect to the standard basis $(z^n)_{n\ge 0}$ for the Hardy space $H^2(\D)$. This representation was first used by
Magnus \cite{M} who found that the Hilbert matrix has no eigenvalues and that its continuous spectrum is $[0,\pi ]$. It was also used in \cite{DS} and  \cite{DJV} to study the Hilbert matrix as an operator on Hardy and Bergman spaces of the disc and in particular to obtain its norm on those spaces.
\par 
The purpose of this paper is to identify and study a multiplicative analogue of $A$. This means that we seek an infinite matrix
with entries $a_{m,n}$ that depend only on the product $mn$ and with properties that parallel those of $A$. Our starting point
is the multiplicative counterpart to \eqref{classical} which we have found to be the integral operator
\begin{equation} \label{multH}
\mathbf{H} f(s) :=\int_{1/2}^{+\infty}f(w)(\zeta(w+s)-1)dw
\end{equation}
acting on Dirichlet series $f(s)=\sum_{n\geq2} a_n n^{-s}.$ Here $\zeta(s)$ denotes the Riemann zeta function, and we
assume that $f$ is in $\Ht_0$, which means that
$$
\|f\|_{\Ht_0}^2:= \sum_{n=2}^{\infty} |a_n|^2 < \infty.
$$
By the Cauchy--Schwarz inequality, every $f$ in $\Ht_0$ represents an analytic function in the half-plane
$\sigma=\Real s>1/2$. The same calculation shows that point evaluations $f \mapsto f(s)$ are bounded linear functionals on $\Ht_0$ for	$s$	in this half-plane. As is readily seen, the reproducing kernel $K_w$ of $\Ht_0$ is $K_{w}(s)=\zeta(s+\overline{w})-1$. This implies that
\begin{equation} \label{inner}
\langle \mathbf{H}f, g\rangle_{\Ht_0}=\int_{1/2}^\infty f(w) \overline{g(w)} dw
\end{equation}
when $f$ and $g$ are Dirichlet polynomials. Now observe that arc length measure on the half-line $(1/2, +\infty)$ is a Carleson measure for $\Ht_0$ (the contribution from $1/2<s<3/2$ is handled by \cite[Theorem~4]{OlS}, while the contribution from $s>3/2$ is handled by a pointwise estimates). We therefore get that \eqref{inner} in fact holds for arbitrary functions $f$ and $g$ in $\Ht_0$, and hence $\textbf{H}$ is well defined and bounded on $\Ht_0$. Taking into
account that every $f$ in $\Ht_0$ is analytic when $\sigma>1/2$, we find that $\langle \mathbf{H}f, f\rangle_{\Ht_0}=0$ if and
only if $f\equiv 0$. Hence \eqref{inner} also implies that $\mathbf{H}$ is a strictly positive operator.
 Now an explicit computation of the integral on the right-hand side of \eqref{multH} shows that the matrix of $\textbf{H}$
with respect to the orthonormal basis $(n^{-s})_{n\ge 2}$ is
\[M:= \left(\frac{1}{\sqrt{mn} \log(mn)} \right)_{m,n\ge 2}.\]
We will refer to this matrix as the multiplicative Hilbert matrix. We will be interested in understanding $M$ as an operator
on $\ell^2=\ell^2(\Nn\setminus\{1\})$, which means that, equivalently, we will be concerned with the properties of the
integral operator $\textbf{H}$ acting on $\Ht_0$.

Our main result reads as follows.
\begin{theorem}\label{basicbound}
	The operator $\mathbf{H}$ is a bounded and strictly positive operator on $\Ht_0$ with $\|\mathbf{H}\|=\pi$. It has no
eigenvalues, and its continuous spectrum is $[0,\pi]$.
\end{theorem}

This theorem, which is in agreement with what is known about the classical Hilbert matrix, should be seen as an outgrowth of
Helson's last two papers \cite{H3, H4}. In these works, a study of multiplicative Hankel matrices was initiated, mainly focused
on the question of to which extent Nehari's theorem \cite{N,P} extends to the multiplicative setting.
We will return to this interesting question in the final section of this paper. At this point, we just wish to point out that the
existence of a canonical operator like $\mathbf{H}$, closely related to the Riemann zeta function, clearly demonstrates that
multiplicative Hankel matrices may arise quite naturally.
\par 
The computation of the norm of $\textbf{H}$ is straightforward, by a simple adaption of the classical proof of
\cite[pp.~226--229]{HLP}. In fact, this adaption leads us to consider an $\ell^p$ version of the multiplicative Hilbert matrix
$M$, namely
\[
 M_p:=\left(\frac{1}{m^{(p-1)/p} n^{1/p} \log (mn)}\right)_{m,n\ge 2},
 \]
where $1<p<\infty$. We will see that $M_p$ has norm $\pi/\sin(\pi/p)$, viewed as an operator on $\ell^p$, which is analogous
to the classical fact that $A$ has norm $\pi/\sin(\pi/p)$ when it acts on $\ell^p$. We will explain this link in Section~\ref{bounded}. This result was actually first obtained by Mulholland \cite{Mulholland}, as a corollary to certain related integral estimates. 
\par 
The identification of the spectrum is the hardest part of  the proof of Theorem~\ref{basicbound}. Inspired by Magnus's work \cite{M},
it is split into two main parts. First, in Section~\ref{Mellin}, we establish estimates near the singular point $s=1/2$ for the anticipated
solutions $f$ to equations of the form
\[ (\textbf{H} -\lambda) f = c\cdot \psi, \]
where $c$ is a constant and $\psi$ is the analytic symbol of $\mathbf{H}$. This means that $\psi$ is the primitive of
$-(\zeta(s+1/2)-1)$ belonging to $\Ht_0$. The point of this estimation is to show that $f'(w)$ must be square integrable
on $(1/2,\infty)$.	Here we make use of the fact that $\zeta(s)-(s-1)^{-1}$ is an entire function, which allows us to
relate $\textbf{H}$ to a classical	 operator studied by Carleman. This analysis requires a fair amount of classical-type
computations involving Mellin transforms. In Section~\ref{spectrum}, we may then finish the proof by resorting to the following
commutation relation, obtained by integration by parts, between $\textbf{H}$ and the differentiation operator $\mathbf{D}$:
\[ \mathbf{D}\mathbf{H}f(s)=-f(1/2)(\zeta(s+1/2)-1)-\mathbf{H}\mathbf{D}f(s).\]

After finishing the proof of Theorem~\ref{basicbound}, we turn to two questions related to Helson's viewpoint, namely
that multiplicative Hankel operators are bona fide (small) Hankel operators on the infinite-dimensional torus $\T^\infty$.
The first question is whether there is a counterpart to the result of \cite{DS, DJV} saying that the norm of $\mathbf{H}_a$
viewed as an operator on $H^p(\D)$ is again $\pi/\sin(\pi/p)$. We will show in Section~\ref{lp} that the analogy with
$\textbf{H}_a$ breaks down at this point, or, more precisely, that $\textbf{H}$ does not extend to a bounded operator
on the $H^p$ analogues of $\Ht_0$, which by Bayart's work \cite{B} can be associated with $H^p(\T^\infty)$. This negative
result is related to, though not a trivial consequence of, the fact that $H^p(\T^\infty)$ is not complemented
in $L^p(\T^\infty)$ \cite{E}.

The final question to be discussed concerns the analytic symbol
\begin{equation} \label{symbol}
	\psi(s):=\sum_{n=2}^{\infty} \frac{n^{-s}}{\sqrt{n}\log n}
\end{equation}
of the multiplicative Hankel matrix. Since $-\psi$ is, up to a linear term, a primitive of the Riemann zeta function,	 it appears	
to be of interest to investigate it more closely. While it is known from \cite{OS} that Nehari's theorem does not hold in the
multiplicative setting, it could still be true that $\psi$ is the Riesz projection of a bounded function.	 In the final Section~\ref{nehari},
we will explain the exact meaning of this statement and show how this question relates to a long-standing embedding problem
for $H^p$ spaces of Dirichlet spaces.

A word on notation: Throughout this paper, the notation $U(z)\lesssim V(z)$ (or equivalently $V(z)\gtrsim U(z)$) means that
there is a constant
$C$ such that $U(z)\le CV(z)$ holds for all $z$ in the set in question, which may be a space of functions or a set of numbers.
We write $U(z)\simeq V(z)$ to signify that both $U(z)\lesssim V(z)$ and $V(z)\lesssim U(z)$ hold.

\section{The norm of the matrix $M_p$}\label{bounded}
In this section, $\| M_p\|_p$ will denote the norm of $M_p$ viewed as an operator on $\ell^p$. Our aim is to prove the
following theorem, which in particular shows that $\| \textbf{H} \| =\pi$.

\begin{theorem}\label{mpnorm}
We have $\| M_p \|_{p} =\pi/\sin(\pi /p)$ for $1<p<\infty$.
\end{theorem}

\begin{proof}
The proof relies, as in \cite[pp.~226--234]{HLP}, on the following homogeneity property of the kernel $(x+y)^{-1}$:
\begin{equation}\label{homo}
 \int_{0}^{\infty} x^{-1/p}\frac{1}{1+x}  dx = \int_{0}^{\infty}x^{-(p-1)/p} \frac{1}{1+x}	 dx=\frac{\pi}{\sin(\pi /p)}.
 \end{equation}
The exact computation of the integral can be found in \cite[p.~254, Example~4]{WW} or \cite[Section~9.5]{D}.
\par 
We prove first that $\| M_p \|_{p} \le \pi/\sin(\pi /p)$. We write $q=p/(p-1)$ and assume that
$(a_m)_{m\geq2}$ is in $\ell^p$ and $(b_n)_{n\geq2}$ is in $\ell^q$. By H\"{o}lder's inequality,
we find that
\[ \sum_{m, n = 2}^\infty |a_m| |b_n| m^{-1/q} n^{-1/p} (\log(mn))^{-1}\le P\cdot Q, \]
where
\begin{equation}\label{PP}
P:= \left(\sum_{m=2}^\infty |a_m|^p\sum_{n\ge 2} n^{-1}
\left (\frac{\log m}{\log n}\right)^{1/q}\frac{1}{\log(mn)}\right)^{1/p}
\end{equation}
and
\begin{equation}\label{QQ}
Q:= \left(\sum_{n=2}^\infty |b_n|^q\sum_{m\ge 2} m^{-1}
\left (\frac{\log n}{\log m}\right)^{1/p}\frac{1}{\log(mn)}\right)^{1/q}.
\end{equation}
By a change of variables argument, each of the inner sums is dominated by the integral in \eqref{homo}, and
hence we obtain the desired bound by duality.

To prove that the norm is bounded below by $\pi/\sin(\pi /p)$, we use the sequences defined by
\[
a_m=m^{-1/p}(\log m)^{-(1+\varepsilon)/p} \qquad \text{and} \qquad
b_{n}=n^{-1/q} (\log n)^{-(1+\varepsilon)/q}
 \]
for which we have
\begin{equation} \label{norms}
\left \| (a_{m})\right\|_p^p=\frac{1}{\varepsilon}+O(1) \qquad \text{and}
\qquad \left\| (b_{n})\right\|_q^q=\frac{1}{\varepsilon}+O(1)
\end{equation}
when $\varepsilon\to 0^+$. We see that
\begin{align*} \sum_{m, n= 2}^\infty a_{m} b_{n} m^{-1/q} n^{-1/p} \frac{1}{\log(mn)}
& = \sum_{m, n= 2}^\infty (\log m)^{-(1+\varepsilon)/p} (\log n)^{-(1+\varepsilon)/q} m^{-1} n^{-1} \frac{1}{\log(mn)} \\
& \ge \int_{\log 3}^{\infty}\int_{\log 3}^{\infty} x^{-(1+\varepsilon)/p}y^{-(1+\varepsilon)/q}\frac{1}{x+y} dxdy.\end{align*}
This iterated integral can computed as the corresponding integral in \cite[p.~233, Equation~9.5.2]{HLP} so that we get
\[
\sum_{m, n= 2}^\infty a_{m} b_{n} m^{-1/q} n^{-1/p} =
\frac{1}{\varepsilon} \left(\frac{\pi}{\sin (\pi /p)} + o(1) \right)
\]
when $\varepsilon\to 0^+$. Combining this estimate with \eqref{norms}, we get the desired bound
$\| M_p\|_p\ge \pi/\sin(\pi/ p)$.
\end{proof}
\par 
It is of interest to observe that when we replace the inner sums in \eqref{PP} and \eqref{QQ} by the respective integrals in \eqref{homo},
we get a strict inequality. In particular, we get that
\[\| \textbf{H} f \|_{\Ht_0} < \pi \| f \|_{\Ht_0}\]
for every nontrivial function $f$ in $\Ht_0$. This means that we have already shown that $\pi$ is not an eigenvalue for $\textbf{H}$.
\par 
Another observation is that the matrix $M_p$ fails to be bounded on $\ell^{p'}$ when $p'\neq p$. This is most easily seen when $p'>p$ because we can find a sequence $a$ in $\ell^{p'}$ for which the entries in $M_{p}a$ become infinite.
When $p'<p$, we can apply the same argument to the conjugate exponents $q$ and $q'$ and the matrix $M_q$.
\par 
In preparation for the proof of the second part of Theorem~\ref{basicbound}, we now clarify the relationship
between $\mathcal{H}^2_0$ and $L^2(1/2,\,\infty)$ implied by Theorem~\ref{mpnorm}.

\begin{corollary} \label{cor:L2emb}
If $f$ is in $\mathcal{H}^2_0$, then $\|f\|_{L^2(1/2,\,\infty)} \leq \sqrt{\pi} \|f\|_{\mathcal{H}^2_0}$. Additionally, $\mathbf{H}$
extends to an operator from $L^2(1/2,\,\infty)$ to $\mathcal{H}^2_0$
and $\|\mathbf{H}f\|_{\mathcal{H}^2_0}\leq\sqrt{\pi}\|f\|_{L^2(1/2,\,\infty)}$.
\end{corollary}
\begin{proof}
	The first statement follows from Theorem~\ref{mpnorm} with $p=2$ and the fact that
	\[\langle \mathbf{H}f,f\rangle_{\mathcal{H}^2_0} = \int_{1/2}^{+\infty} |f(w)|^2\,dw.\]
	Given $f \in L^2(1/2,\,\infty)$, clearly $\mathbf{H}f$ is a Dirichlet series vanishing at $+\infty$. If
$g(s) = \sum_{n\geq2} b_n n^{-s}$, it follows from Fubini's theorem that
\[
\langle \mathbf{H}f,g\rangle_{\mathcal{H}^2_0} = \sum_{n=2}^\infty \left(\int_{1/2}^\infty f(w)n^{-w}\,dw\right)\overline{b_n}
=  \int_{1/2}^\infty f(w)\overline{g(w)}\,dw,
\]
so that \eqref{inner} extends to hold for $f \in L^2(1/2,\,\infty)$ and Dirichlet polynomials $g$. The second statement now
follows from the first, since
\[
\|\mathbf{H}f\|_{\mathcal{H}^2_0} = \sup_{\|g\|_{\mathcal{H}^2_0}
=1} \left|\langle \mathbf{H}f, g \rangle_{\mathcal{H}^2_0}\right|
\leq\sup_{\|g\|_{\mathcal{H}^2_0}=1}\|f\|_{L^2(1/2,\,\infty)}\|g\|_{L^2(1/2,\,\infty)} \leq \sqrt{\pi} \|f\|_{L^2(1/2,\,\infty)}. \qedhere
\]	
\end{proof}

\section{Estimates for solutions of $(\mathbf{H}-\lambda)f=c\psi$}\label{Mellin}
In preparation for the characterization of the spectrum of $\mathbf{H}$, we will in this section prove precise asymptotics
as $s \to 1/2$ for solutions $f$	in $\mathcal{H}_0^2$ of the equation $(\mathbf{H}-\lambda)f=c\psi$, where $c$ is a constant
and $\psi$ is the analytic symbol of $\mathbf{H}$ defined by \eqref{symbol}.
 The considerations to come are in fact of a rather general nature, providing a spectral decomposition of $f$ in terms of
 generalized eigenvectors of the (shifted) Carleman operator \cite[p. 169]{C} defined by
\[\mathbf{C}f(s) = \int_{1/2}^\infty \frac{f(w)}{s+w - 1}\,dw, \qquad s > 1/2.\]
We choose to focus on $\mathbf{H}$ for simplicity, but it will be clear from the proof of the next theorem that minor modifications
yield similar results for other integral operators whose kernels are perturbations $K(s+w)$, $K$ analytic, of the Carleman kernel.
\begin{theorem} \label{thm:rep}
Suppose that $0 < \lambda < \pi$, and let $\psi$ denote the analytic symbol of \ $\mathbf{H}$, that is
$$\psi(s) = \sum_{n=2}^\infty \frac{1}{\sqrt{n} \log n} n^{-s}, \qquad \Real s > 1/2.$$ If $f$ in $ \mathcal{H}_0^2$ satisfies	
$(\mathbf{H}-\lambda)f=c\psi$, then there exists a complex number $d$ and polynomially bounded sequences of complex numbers
$(c_k)_{k\geq1}$ and $ (d_k)_{k\geq1}$ such that $f$ has the series representation
\begin{equation} \label{eq:thmrep}
f(s) = cd + \sum_{k=1}^\infty (s-1/2)^{2k-1/2}\left( c_{k}(s-1/2)^{-i\theta} + d_{k}(s-1/2)^{i\theta}\right), \qquad 1/2 < s < 3/2,
\end{equation}
where $\theta$ is a real number dependent on $\lambda$, namely
\[
\theta = \frac{1}{\pi}\log\left(\frac{\pi}{\lambda}-\sqrt{\left(\frac{\pi}{\lambda}\right)^2-1}\right).
\]
In particular, if $f$	in $\mathcal{H}^2_0$ solves $(\mathbf{H}-\lambda)f=c\psi$ then $f'\in L^2(1/2,\,\infty)$.
\end{theorem}
\begin{remark}
Note that for each $k$, the functions $s \mapsto (s-1/2)^{2k-1/2 \pm i\theta}$  are generalized eigenvectors of the Carleman
operator $\mathbf{C}$ belonging to the eigenvalue $\lambda$, $0 < \lambda < \pi$; see Lemma \ref{lem:mellincarleman}. The
constant function $s \mapsto cd$ is not such an eigenfunction, and its appearance in \eqref{eq:thmrep} will allow us to derive
a contradiction in the case that $c \neq 0$.

 It is also possible to treat the case $\lambda = \pi$ with the methods below, although we choose not to since we do not need it.
 Carrying out the details, one obtains for $\lambda = \pi$ a decomposition of $f$ in terms of the eigenfunctions
 $s \mapsto (s-1/2)^{2k-1/2}$ and  $s \mapsto (s-1/2)^{2k-1/2} \log(s-1/2)$ of the Carleman operator.
\end{remark}

To simplify the computations and to align our proof with the classical representation of the Carleman operator, we will in this
section shift everything to $\mathbb{R}_+=(0,\infty)$, and prove	 Theorem~\ref{thm:rep} on this ray. Shifting the representation
back to $(1/2,+\infty)$ will then give \eqref{eq:thmrep}. This means that we consider $\mathcal{H}_0^2$ the space of Dirichlet series
\[
f(s) = \sum_{n=2}^\infty \frac{a_n}{\sqrt{n}}\, n^{-s},
\]
with coefficients $(a_n)_{n\geq2} \in \ell^2$, and the operator
\[\mathbf{H} f(s) = \int_0^\infty f(w)\left(\zeta(s+w+1)-1\right)\,ds.\]
We let $\{x\}$ denote the fractional part of $x$, and use the well-known formula
\[
\zeta(s+1)-1 = \frac{1}{s}-(s+1)\int_1^\infty \{x\} x^{-(s+1)}\,\frac{dx}{x} =
 \frac{1}{s}-(s+1) \int_0^\infty \{e^x\}e^{-(s+1)x}\,dx =: \frac{1}{s}-K(s).
 \]
The function $1/s$ is the kernel of Carleman's operator, defined on $L^2(\mathbb{R}_+)$ as
\[\mathbf{C}f(s) = \int_0^\infty \frac{f(w)}{s+w}\,dw.\]
We will let $\mathbf{K}$ denote the similarly defined integral operator with kernel $(s,w) \mapsto K(s+w)$, so that
$\mathbf{H} = \mathbf{C}-\mathbf{K}$. For $0 < \lambda < \pi$ and $f$	 in $\mathcal{H}_0^2$, we consider the
equation $(\mathbf{H}-\lambda)f = c \psi$, where $\psi$ denotes
$$\psi(s) = \sum_{n=2}^\infty \frac{1}{n \log n} n^{-s}.$$
(Note that this function also differs by a $1/2$ shift from the actual symbol appearing in Theorem~\ref{thm:rep}.)
It is convenient to rewrite this equation in the form
\begin{equation} \label{eq:eigeneq}
(\mathbf{C}-\lambda)f = \mathbf{K}f+c \psi.
\end{equation}
To analyze the equation \eqref{eq:eigeneq}, we will use the Mellin transform, which is defined by
\begin{equation} \label{eq:mellinxform}
	\mathcal{M}f(z) = \int_0^\infty s^{z}f(s)\,\frac{ds}{s}.
\end{equation}
By the Cauchy--Schwarz inequality and Corollary~\ref{cor:L2emb}, taking into account the rapid decay
near infinity, we obtain that if $f$ is in $\mathcal{H}_0^2$, then the integral \eqref{eq:mellinxform} converges absolutely
when $\Real z>1/2$. This means that the function $\mathcal{M}f(z)$ is analytic in (at least) $\Real z>1/2$. Our first
goals are thus to compute $\mathcal{M}\mathbf{C}f$ and $\mathcal{M}\mathbf{K}f$ for $f $ in $\mathcal{H}_0^2$,
as well as the special transform $\mathcal{M} \psi$.

\begin{lemma} \label{lem:mellincarleman}
	Suppose that $f$ is  in $\mathcal{H}_0^2$. Then
	\begin{equation} \label{eq:mellincarleman}
		(\mathcal{M}\mathbf{C}f)(z) = \frac{\pi}{\sin{(\pi z)}} \,(\mathcal{M}f)(z),
	\end{equation}
	has a meromorphic continuation to $\Real z>1/2$.
	\begin{proof}
		When $\Real z<1$, $z \not \in \mathbb{Z}$ and $w>0$, we have
		\[\int_0^\infty \frac{s^{z-1}}{s+w}\,ds = \frac{\pi}{\sin{(\pi z)}}w^{z-1},\]
		which is the same integral \eqref{homo} which was used in the proof of Theorem~\ref{mpnorm}. By this formula and
Fubini's theorem, we obtain \eqref{eq:mellincarleman} in the strip $1/2<\Real z<1$. However, the right hand
side of \eqref{eq:mellincarleman} has a meromorphic continuation to the domain $\Real z>1/2$.
	\end{proof}
\end{lemma}
\begin{remark}
Note that the choice of $\theta$ is such that $\pi/\sin \left( \pi (i\theta + 1/2) \right) = \lambda$. This motivates the appearance
of the functions $s \mapsto s^{2k-1/2 \pm i\theta}$ in \eqref{eq:thmrep} as generalized eigenfunctions to the Carleman
operator. Compare with the remark following Theorem \ref{thm:rep}.
\end{remark}
\begin{lemma} \label{lem:mellinK}
	Let $f$ be a function	in $ \mathcal{H}_0^2$. Then $(\mathcal{M}\mathbf{K}f)(z)$ has a meromorphic continuation to
$\Real z<1$ with simple poles at the non-positive integers. If\,  $\Real z \le 1-\varepsilon$ and $|\Imag z|\ge \varepsilon$,
for some positive $\varepsilon$, then 
	\begin{equation} \label{eq:growth}
		(\mathcal{M}\mathbf{K}f)(z) \lesssim \|f\|_{\mathcal{H}_0^2} |z| e^{-\pi|\Imag z|/2}.
	\end{equation}
	\begin{proof}
		We begin by computing
\begin{equation} \label{eq:Kfcomp}
\mathbf{K}f(s) = \int_0^\infty f(w)K(s+w)\,dw = \sum_{n=2}^\infty \frac{a_n}{\sqrt{n}\log{n}}\left(\alpha_n(s)+\beta_n(s)\right),
\end{equation}
where
\begin{align*}
\alpha_n(s) &= \int_0^\infty A_n(x)\,se^{-sx}\,xdx,\qquad &A_n(x)& = \frac{1}{1+x/\log{n}}\,\frac{\{e^x\}}{x}e^{-x}, \\
\beta_n(s) &= \int_0^\infty 2B_n(x)\,e^{-sx}\,xdx,\qquad &B_n(x)& =
\frac{1}{2}\left(\frac{1}{(1+x/\log{n})^2}+\frac{1}{1+x/\log{n}}\right)\,\frac{\{e^x\}}{x}e^{-x}.
\end{align*}
We will only need the estimates $A_n(x),\,B_n(x) \leq e^{-x}$, which imply that $\mathbf{K}f(s)$ is analytic in $\Real s>-1$,
since $(a_n/(\sqrt{n}\log{n}))_{n\geq2}$ is in $\ell^1$. We apply the Mellin transform of \eqref{eq:Kfcomp}, initially
with $0 < \Real z < 1$, obtaining		
\[
(\mathcal{M}\mathbf{K}f)(z) =\sum_{n=2}^{\infty} \frac{a_n}{\sqrt{n}\log{n}}\left(\Gamma(1+z)\widetilde{\alpha}_n(z)
+\Gamma(z)\widetilde{\beta}_n(z)\right),
\]
where $\Gamma$ denotes the Gamma function and
\[
\widetilde{\alpha}_n(z) = \int_0^\infty A_n(x)\,x^{1-z}\,\frac{dx}{x} \qquad \text{and} \qquad
\widetilde{\beta}_n(z) = \int_0^\infty 2B_n(x)\,x^{2-z}\,\frac{dx}{x}.
\]
When $\Real z<1$, we use the estimates $A_n(x),\,B_n(x)\leq e^{-x}$ along with the triangle inequality to obtain
\[
|\widetilde{\alpha}_n(z)| \leq \Gamma(1-\Real z) \qquad \text{and} \qquad
|\widetilde{\beta}_n(z)| \leq 2 \Gamma(2-\Real z).
\]
Hence $\mathcal{M}\mathbf{K}f$ has a meromorphic continuation to $\Real z<1$, with simple poles at the poles of $\Gamma(z)$.
Moreover, by the Cauchy--Schwarz inequality, we obtain that
\[
|(\mathcal{M}\mathbf{K}f)(z)| \lesssim \|f\|_{\mathcal{H}_0^2}\big(|\Gamma(1+z)|\Gamma(1-\Real z)+
2|\Gamma(z)|\Gamma(2-\Real z)\big).
\]
When $|\Imag z|\ge \varepsilon$, we may use the functional equation and reflection formula for the Gamma function, and
estimate further that
\begin{equation}\label{eq:bound1}
|(\mathcal{M}\mathbf{K}f)(z)| \lesssim \|f\|_{\mathcal{H}_0^2}\,\big(\left|\Gamma(1+z)\right|\Gamma(1-\Real z)\big)
=\|f\|_{\mathcal{H}_0^2}\, \frac{\pi}{|\sin{(\pi z)}|}\,\frac{\Gamma(1-\Real z)}{|\Gamma(-z)|}.
\end{equation}
By our restriction that $\Real z \le 1-\varepsilon$ and $|\Imag z|\ge \varepsilon$, Stirling's
formula (see \cite[p. 525]{MV})
now yields that
\[
\frac{\Gamma(1-\Real z)}{|\Gamma(-z)|}	\lesssim
\frac{|1-\Real z|^{1/2-\Real z}}{|z|^{-\Real z-1/2}} e^{\pi |\Imag z|/2} \lesssim |z|		
e^{\pi |\Imag z|/2},
\]
where the implicit constants depend only on $\varepsilon$.
Hence returning to \eqref{eq:bound1}, we find that
\[|(\mathcal{M}\mathbf{K}f)(z)| \lesssim \|f\|_{\mathcal{H}_0^2} |z| e^{-\pi|\Imag z|/2} \]
as claimed.
\end{proof}
\end{lemma}

\begin{lemma} \label{lem:Melling}
	For $\Real z > 0$, we have
	\begin{equation} \label{eq:Melling}
		\mathcal{M}\psi(z) = -\frac{1}{z^2} + \sum_{n=0}^\infty \frac{b_n}{z+n} + E_\psi(z),
	\end{equation}
	where $|b_n|$ decays super-exponentially, and $E_\psi(z)$ is an entire function that, for every real number $R$, is bounded
in the half-plane $\Real z < R$. Hence $\mathcal{M}\psi(z)$ has a meromorphic continuation to $\mathbb{C}$ with
a double pole at $z=0$ and simple poles at the negative integers.
	\begin{proof}
	Set $h(s): = \psi(s) - \log s$. Since $h'(s) = \zeta(s+1) - 1 - 1/s$, $h(s) = \sum_{n\geq0} b_n s^n$ is an entire function.
Note now that for $\Real z > 0$ we have
	$$\int_0^1 s^{z-1} \log s \, ds = -\frac{1}{z^2},$$
	while
	$$\int_0^1 s^{z-1}h(s) \, ds = \sum_{n=0}^\infty \frac{b_n}{z+n}.$$
	We finish the proof by setting $E_\psi(z) := \int_1^\infty s^{z-1} \psi(s) \, ds$.
	\end{proof}
\end{lemma}

\begin{proof}[Proof of Theorem~\ref{thm:rep}]
	Suppose that $0 < \lambda < \pi$. Transforming the equation \eqref{eq:eigeneq} by the Mellin transform and solving
for $\mathcal{M}f$, we obtain
	\begin{equation}\label{eq:Mfcomp}
\mathcal{M}f(z) = \left( \frac{\pi}{\sin (\pi z)} - \lambda\right)^{-1}\left(\mathcal{M}\mathbf{K}f(z)+c\mathcal{M}\psi(z)\right).
	\end{equation}
	Initially this formula is only valid for $1/2 < \Real z < 1$, but we note that the left hand side can be analytically
continued to $\Real z >1/2$ and the right hand side can be meromorphically continued to $\Real z < 1$.
	
	The inverse Mellin transform is given by
	\begin{equation} \label{eq:inverseMellin}
		\mathcal{M}^{-1}h(s) = \frac{1}{2\pi i}\int_{\kappa-i\infty}^{\kappa+i\infty} s^{-z}h(z)\,dz
	\end{equation}
for a suitable $\kappa$. For \eqref{eq:Mfcomp} the Mellin inversion theorem allows us to choose $\kappa \in (1/2,1)$. Our expressions for $\mathcal{M}\mathbf{K}f$ and $\mathcal{M}\psi$ show that the right-hand side of \eqref{eq:Mfcomp} is meromorphic in $\Real z<1$ with (possible) simple poles at the solutions of $\sin(\pi z) = \pi/\lambda$ as well as at $z = 0$. Note here that the factor in front of $\mathcal{M}\mathbf{K}f(z)+c\mathcal{M}\psi(z)$ has simple zeroes at the integers.
Note also that there actually are no poles in $\Real z>1/2$, since $\mathcal{M}f(z)$ is analytic there. Hence we are left with
the pole $z=0$ (if $c\neq0$) and those given by
	\[1-\frac{\lambda}{\pi}\sin{(\pi z)} = 0, \qquad \Real z\leq 1/2 \qquad\Longleftrightarrow\qquad z = \pm i\theta+(2k+1/2),\]
	where $k=0,\,-1,\,-2,\,\ldots$
		
	We now compute \eqref{eq:inverseMellin} for $h = \mathcal{M}f$ and $\kappa = 2/3$ by the method of residues.
	Let $J_n=[\theta]+n$ and form the rectangular contour $\mathcal{J}_n$ with corners in $2/3\pm iJ_n$ and $-(2J_n+3/2)\pm iJ_n$,
traversed counter-clockwise. Using	\eqref{eq:growth} and \eqref{eq:Melling}, straightforward estimates show that for $0 < s < 1$
we have
\[
\lim_{n \to \infty} \int_{\mathcal{J}_n}s^{-z}\mathcal{M}f(z)\,dz = \frac{1}{2\pi i}\int_{2/3-i\infty}^{2/3+i\infty} s^{-z}
\mathcal{M}f(z)\,dz.
\]
Evaluating the left-hand side by residues, we obtain
	\[f(s) = cd + \sum_{k=0}^\infty s^{2k-1/2}\left( c_{k}s^{-i\theta} + d_{k}s^{i\theta}\right), \qquad 0 < s < 1,\]
	where $cd$, $c_k$, and $d_k$ are obtained as the residues of the right-hand side of \eqref{eq:Mfcomp} at $z = 0$,
$z = i\theta -2k + 1/2$ and	 $z = -i\theta -2k + 1/2$, respectively. In fact, it is clear that $c_{k}$ and $d_{k}$ grow at
most polynomially in $k$, as seen from the estimates of Lemma~\ref{lem:mellinK} and Lemma~\ref{lem:Melling}.
	
	It remains to show that $c_0 = d_0 = 0$. However, either of them assuming a non-zero value contradicts the fact that $f$
is in $ L^2(\mathbb{R}_+)$. Moving back to $(1/2,+\infty)$, we obtain \eqref{eq:thmrep}.
	
The final statement follows from the fact that $f'(s)$ is bounded in $1/2<s<1$ due to \eqref{eq:thmrep}, the
contribution from $s>1$ is easily estimated by the fact that $f$ is a Dirichlet series in $\mathcal{H}^2_0$. 
\end{proof}

Note that in the excluded case $\lambda = \pi$ one may use the same argument, but the representation of $f$ is different
because all poles of the right-hand side of \eqref{eq:Mfcomp} except $z = 0$ are double. We also note that a more careful
analysis would show that the sequences $(c_k)_{k\geq0}$ and $(d_k)_{k\geq0}$ are in fact bounded, but since we do not
need this, we have not made an effort to optimize this part of the theorem.

\section{The spectrum of the multiplicative Hilbert matrix}\label{spectrum}
In this section we establish that $\mathbf{H}$ has the purely continuous spectrum $[0, \pi]$ on $\mathcal{H}_0^2$. Our
argument is based on a commutation relation between $\mathbf{H}$ and the operator $\mathbf{D}$ of differentiation,
$\mathbf{D}f(s)=f'(s)$. To establish this relation, we observe that
\[ \mathbf{D}\mathbf{H}f(s) = \int_{1/2}^{\infty} f(w) \mathbf{D} (\zeta(w+s)-1) dw, \qquad s > 1/2. \]
Supposing that $f'$ is integrable on the segment $(1/2,1)$, we get that
\begin{align*}
	\mathbf{D}\mathbf{H}f(s)&=-f(1/2) (\zeta(s+1/2)-1) - \int_{1/2}^\infty f'(w) (\zeta(w+s)-1) dw \\
	&=-f(1/2) (\zeta(s+1/2)-1) - \mathbf{H}\mathbf{D}f(s), \qquad s > 1/2,
\end{align*}
where we have defined $f(1/2)=f(1)-\int_{1/2}^1 f'(w) dw$. Thus, $\mathbf{D}$ and $\mathbf{H}$ anti-commute up to
an (unbounded) rank-one term.  This observation is crucial for the characterization of the spectrum of $\mathbf{H}$.

To demonstrate that $\mathbf{H}$ has the purely continuous spectrum $[0, \pi]$, it suffices to show that $\mathbf{H}$ has
no eigenvalues and that $H - \lambda$ does not have full range for $\lambda$ in $(0,\pi)$. Indeed, $\mathbf{H}$ is a
positive operator with norm $\pi$, and so it follows that its spectrum is $[0, \pi]$. Since any $\lambda$ in the spectrum of a
self-adjoint operator must either be an eigenvalue or part of the continuous spectrum, we can conclude that $\mathbf{H}$ has
purely continuous spectrum. With this in mind we now finish the proof of Theorem \ref{basicbound}.

\begin{theorem} \label{thm:pp}
The operator $\mathbf{H} : \mathcal{H}_0^2 \to \mathcal{H}_0^2$ has no point spectrum. Furthermore,
if $f $ in $ \mathcal{H}_0^2$ solves the equation $(\mathbf{H}-\lambda)f = c \psi$, where $c$ is a complex number and
$$\psi(s) = \sum_{n=2}^\infty \frac{1}{\sqrt{n} \log n} n^{-s},$$
then $f = c = 0$. In particular, the spectrum of\/ $\mathbf{H}$ is $[0,\pi]$ and purely continuous.
\end{theorem}
\begin{proof}
We have already proved that $\lambda = 0$ and $\lambda = \pi$ are not eigenvalues, since we have shown in Section \ref{bounded}
that $\mathbf{H}$ is a strictly positive operator for which $\| \textbf{H} f \|_{\Ht_0} < \pi \| f \|_{\Ht_0}$, $f \neq 0$. It is hence
sufficient to verify the second part of Theorem \ref{thm:pp}, since it shows  simultaneously that no $\lambda$ in
$(0, \pi)$ is an eigenvalue, and that $\mathbf{H} - \lambda$ does not have full range.

Accordingly, we suppose that $f$ in $ \mathcal{H}_0^2$ satisfies $(\mathbf{H}-\lambda)f = c \psi$. By Theorem \ref{thm:rep}, we
have the series representation \eqref{eq:thmrep}. In particular $f'$ is square-integrable on $(1/2, \infty)$ and $f(1/2) = cd$. But
noting that $\psi'(s) = \zeta(s+1/2) - 1$ and using the commutation relation of $\mathbf{H}$ and $\mathbf{D}$, we then	get that
$$-(\mathbf{H}+\lambda)f' - cd(\zeta(s+1/2) -1) = c(\zeta(s+1/2) - 1).$$
Since $f'$ is in $L^2(1/2,\infty)$ we use Corollary~\ref{cor:L2emb} to conclude that $\mathbf{H}f'$ is also
in $L^2(1/2,\infty)$. Since $\zeta(s+1/2)$ has a pole of order $1$ at  $s = 1/2$, it follows that $d = -1$. Hence, we have
obtained that
\begin{equation} \label{eq:poseig}
(\mathbf{H}+\lambda)f' = 0.
\end{equation}
From \eqref{eq:poseig} and Corollary~\ref{cor:L2emb}, we get that $f'$ is $ \mathcal{H}_0^2$. But since $\mathbf{H}$ is a
positive operator on $\mathcal{H}^2_0$, applying \eqref{eq:poseig} again, we find that $f' \equiv 0$.
\end{proof}

\section{Failure of boundedness of $\textbf{H}$ on $\Hp_0$ when $p\neq 2$}\label{lp}
We follow \cite{B}	and define ${\Hp}$ as the completion of the set of Dirichlet polynomials $P(s)=\sum_{n\leq N} a_n n^{-s}$ with
respect to the norm
\[\Vert P\Vert_{{\Hp}}:=\left(\lim_{T\to \infty} \frac{1}{T}\int_{0}^T \vert P(it)\vert^{p}dt\right)^{1/p}. \]
The Dirichlet series of a function $f$	in ${\Hp}$ converges uniformly in each half-plane 	
$\Real s>1/2+\varepsilon$, $\varepsilon>0$, so $f$ is analytic in the half-plane $\Real s>1/2$ (see \cite{B,QQ}). The space $\Hp_0$ is the subspace of $\Hp$ consisting of Dirichlet series of the form $\sum_{n\geq2} a_n n^{-s}$, which means
that series in $\Hp_0$ vanish at $+\infty$.

\begin{theorem}\label{pth}
	$\mathbf{H}$ does not act boundedly on $\Hp_0$ for $1\leq p<\infty$, $p\neq2$.
\end{theorem}

The proof of this theorem requires us to associate $\Hp$ with $H^p(\T^\infty)$. This means that we need to invoke the so-called
Bohr lift, which we now recall (see \cite{HLS, QQ} for further details). For every positive integer $n$, the fundamental theorem
of arithmetic allows the prime factorization
\[n = \prod_{j=1}^{\pi(n)} p_j^{\kappa_j},\]
which associates $n$ to the finite non-negative multi-index $\kappa(n) = (\kappa_1,\,\kappa_2,\,\kappa_3,\,\ldots\,)$. The Bohr
lift of the Dirichlet series
$f(s)=\sum_{n\geq1} a_n n^{-s}$ is the power series
\begin{equation} \label{eq:bohrlift}
	\mathcal{B}f(z) = \sum_{n=1}^\infty a_n z^{\kappa(n)},
\end{equation}
where $z = (z_1,\,z_2,\,z_3,\,\ldots\,)$. Hence \eqref{eq:bohrlift} is a power series in infinitely many variables, but each term
contains only a finite number of these variables. An important example is the Bohr lift of the Riemann zeta function. Let $f_w(s) = \zeta(s+w)$ for $\Real(w)>1/2$. Using the Euler product of the Riemann zeta function, we find that
\begin{equation} \label{eq:zetalift}
	\mathcal{B}f_w(z) = \sum_{n=1}^\infty n^{-w}z^{\kappa(n)} = \prod_{j=1}^\infty \left(1-p_j^{-w}z_j\right)^{-1}.
\end{equation} Indeed, any Dirichlet series with an Euler product has a Bohr lift that separates the variables in the same way.

Under the Bohr lift, $\mathcal{H}^p$ corresponds to the Hardy space $H^p(\T^\infty)$, which we view as a subspace of
$L^p(\mathbb{T}^\infty)$. This means that $\mathcal{B}$ is a multiplicative and isometric map from $\Hp$ onto $H^p(\T^\infty)$.
We refer to \cite{B, CG, HLS, QQ} for the details, mentioning only a few important facts. Functions in $H^p(\mathbb{T}^\infty)$
are analytic at the points $\xi \in \mathbb{D}^\infty \cap \ell^2$. Indeed the reproducing kernel at $\xi$ is given by
\[K_\xi(z) = \prod_{j=1}^\infty \left(1-\overline{\xi_j}z_j\right)^{-1},\]
compare with \eqref{eq:zetalift}. The Haar measure of the compact abelian group $\mathbb{T}^\infty$ is simply the product of the
normalized Lebesgue measures for each variable. In particular, $H^p(\T^d)$ is a natural subspace of $H^p(\T^\infty)$. We denote the
orthogonal projection (Riesz projection) from $L^2(\T^{\infty})$ onto $H^2(\T^{\infty})$ by $P_+$. Even though $H^p(\T^\infty)$
is uncomplemented in $L^p(\T^\infty)$ when $p\neq 2$ \cite{E}, we can still identify its dual with the Riesz projection of $L^q(\T^\infty)$
 for $1/p+1/q=1$ using the Hahn--Banach theorem, $(H^p(\mathbb{T^\infty}))^\ast = P_+ L^q(\mathbb{T^\infty})$, $1\leq p < \infty$. 

We require the following lemma which is established by direct computation. Here and in what follows, the $L^p$ norm with
respect to normalized Lebesgue measure on $\T$ (or $\T^\infty$) is denoted by $\| \cdot \|_p$.

\begin{lemma}\label{simple}
Let $\lambda$ be a real parameter and suppose that $0<\varepsilon(1+|\lambda|)<1/4$, $1\le p< \infty$. Then
\[
\|1+\varepsilon (z+\lambda \overline{z}) \|_{p}^p=1+\frac{p}{4} \left[(p-1)(1+\lambda)^2+(1-\lambda)^2\right]\varepsilon^2
 +O(\varepsilon^3).
 \]
The norm is minimal when $\lambda=(2-p)/p$:
\[
\left\|1+\varepsilon \left(z+\frac{(2-p)}{p} \overline{z}\right) \right\|_{p}^p= 1+(p-1)\varepsilon^2 +O(\varepsilon^3).
\]
\end{lemma}

\begin{proof}
	We write $z=e^{i\theta}$ so that we have
	\begin{align*}
|1+\varepsilon (z+\lambda \overline{z})|^p&=\left(1+2\varepsilon(1+\lambda)\cos \theta+\varepsilon^2(1+\lambda)^2\cos^2 \theta+
\varepsilon^2 (1-\lambda)^2 \sin^2\theta\right)^{p/2} \\
&= 1+p \varepsilon(1+\lambda)\cos \theta \\
 &\qquad +\frac{p}{2}\varepsilon^2\left[1+2\Big(\frac{p}{2}-1\Big)\right](1+\lambda)^2\cos^2 \theta+
 \frac{p}{2}\varepsilon^2 (1-\lambda)^2 \sin^2\theta +O(\varepsilon^3).
\end{align*}
Integrating, we get
\[
\|1+\varepsilon (z+\lambda \overline{z}) \|_p^p=1+\frac{p}{4} \left[(p-1)(1+\lambda)^2+(1-\lambda)^2\right]\varepsilon^2
+O(\varepsilon^3). \qedhere
\]
\end{proof}
The point of the lemma is that $p^2/4>p-1$ whenever $p\neq 2$, so that (one-dimensional) Riesz projection acts expansively
on $g(z)=1+\varepsilon (z+\lambda \overline{z})$, since $\|P_+g\|_p^p = 1+(p/2)^2\varepsilon^2+O(\varepsilon^4)$. 

\begin{proof}[Proof of Theorem~\ref{pth}]
	Assume first that $p>1$. To estimate the norm of $\mathbf{H}$ on $\mathcal{H}^p_0$ from below, we will choose $G$
in $L^q(\T^\infty)$ with $1/p+1/q=1$ such that $G(0)=1$. Then using that $\zeta(s+w)-1$ is the reproducing kernel of
$\mathcal{H}^2_0$, we get for $f \in \mathcal{H}^p_0$ that
\[\langle \mathcal{B}\mathbf{H}f,G\rangle_{L^2(\mathbb{T}^\infty)}=\langle \mathbf{H}f,\mathcal{B}^{-1}P_+
G\rangle_{\mathcal{H}^2}= \int_{1/2}^\infty f(w) \overline{\left(\B^{-1} P_+G(w)-1\right)}\, dw.
\]
Specifically, we set
\[
G(z)=\prod_{j=1}^{\infty} \left(1+\frac{2}{q}p_j^{-\alpha}\left(z_j +\frac{(2-q)}{q}\overline{z_j}\right)\right)
\]
	where $\alpha>1/2$. Using Lemma~\ref{simple} we find that
\[
\|G\|_q^q = \prod_{j=1}^\infty \left\|1+\frac{2}{q}p_j^{-\alpha} \left(z_j+\frac{(2-q)}{q} \overline{z_j}\right) \right\|_q^q =
\prod_{j=1}^\infty \left(1+\frac{4(q-1)}{q^2} p_j^{-2\alpha} +O(p_j^{-3\alpha})\right).
\]
	To estimate the Euler products $\prod_{j\geq1} (1+ \lambda p_j^{-s})$ for, say $1<s<2$, we use that
\[
\prod_{j=1}^\infty (1+\lambda p_j^{-s}) = \prod_{j=1}^\infty \frac{(1+\lambda p_j^{-s})\big(1-\lambda p_j^{-s} +
O(p_j^{-2s})\big)}{(1-p_j^{-s})^\lambda} \simeq \zeta(s)^\lambda \simeq (s-1)^{-\lambda}.
\]
	We get that $\|G\|_q \simeq (2\alpha-1)^{-4/(pq^2)}$ as $\alpha \to 1/2$, since $(q-1)/q=1/p$. If $1/2 <\alpha, w < 1$, then
\[
\mathcal{B}^{-1}P_+G(w)= \prod_{j=1}^{\infty} \left(1+(2/q) p_j^{-\alpha-w}\right) \simeq (\alpha+w-1)^{-2/q}.
\]
	We now choose
\[
f(w) = \prod_{j=1}^\infty \left(1 + (2/p) p_j^{-\alpha-w}\right) -1 \simeq (\alpha+w-1)^{-2/p}.
\]
	The norm of $f$ can be computed as in the proof of Lemma~\ref{simple},
\[
\|\mathcal{B}f\|_p = \prod_{j=1}^\infty \left\|1 + (2/p)p_j^{-\alpha} z_j\right\|_p = \prod_{j=1}^\infty \left(1+ p_j^{-2\alpha} +
 O(p_j^{-4\alpha})\right)^\frac{1}{p} \simeq (2\alpha-1)^{-1/p}.
 \]
	Combining everything, we get that
\[
\frac{\left|\langle \mathcal{B}\mathbf{H}f,G\rangle_{L^2(\mathbb{T}^\infty)}\right|}{\|\mathcal{B}f\|_p\|G\|_{q}} \gtrsim
 (2\alpha-1)^{4/(pq^2)+1/p} \int_{1/2}^1 (\alpha+w-1)^{-2}\,dw \simeq (2\alpha-1)^{4/(pq^2)+1/p-1}.
 \]
	The exponent is negative if $p\neq 2$ since, in this case, $pq>4$ so letting $\alpha \to 1/2$ shows that $\mathbf{H}$ is
unbounded on $\mathcal{H}^p_0$.
	
For $p=1$, we make a minor adjustment. We can use the same $f$ (with $p=1$), but we choose
\[
G(z) = \prod_{j=1}^\infty \left(1+(1/4) p_j^{-\alpha}(z_j-\overline{z_j})\right).
\]
The point is that $z_j - \overline{z_j} = 2i\sin(\theta_j)$, if $z_j = e^{i\theta_j}$, so we get that
\[
\|G\|_\infty = \prod_{j=1}^\infty \sqrt{1 + (p_j^{-\alpha}/2)^2} = \prod_{j=1}^\infty \left(1 + (1/8)p_j^{-2\alpha} +
O(p_j^{-4\alpha})\right) \simeq (2\alpha-1)^{-1/8}.
\]
The rest of the argument works like above, the conclusion coming from that $1/8 - 1/4 < 0$.
\end{proof} 

\section{Symbols of the multiplicative Hilbert matrix}\label{nehari}
To place our discussion in context, we begin with some general considerations concerning Hankel forms, i.e., the bilinear forms
associated with (additive or multiplicative) Hankel matrices. We recall that any function $\psi$ in $H^2(\T)$ defines a Hankel
form $H_{\psi}$ by the relation
\[
H_{\psi}(f,g)=\langle fg, \psi \rangle_{L^2(\T)},
\]
which makes sense at least for polynomials $f$ and $g$. Nehari's theorem \cite{N} says that $H_{\psi}$ extends to a bounded form
on $H^2(\T)\times H^2(\T)$ if and only if $\psi=P_+ \varphi$ for a bounded function $\varphi$ in $L^\infty (\T)$. Moreover, 
$\| H_{\psi} \|= \| \varphi \|_\infty$ if we choose $\varphi$
to have minimal $L^\infty$ norm. By the Hahn-Banach theorem and the observation that
\[\langle f, \varphi \rangle_{L^2(\T)} = \langle f, P_+\varphi \rangle_{L^2(\T)},\]
at least for polynomials $f$, we note an equivalent formulation of the first part of Nehari's theorem: $H_{\psi}$ defines
a bounded form if and only if $\psi$ induces a bounded functional on $H^1(\T)$, in the sense that there exist $C > 0$ such that
for every polynomial $f$ it holds that $|\langle f, \psi \rangle_{L^2(\T)}| \leq C\|f\|_{1}$.  See for example
\cite[Section~1.4]{Nikolskii}.

In this context let us indicate an alternative proof (in fact, the original approach of Hilbert) of the fact that the usual Hilbert
matrix has norm $\pi$. Let $\varphi(\theta) = ie^{-i\theta}(\pi - \theta)$, $\theta \in [0, 2 \pi)$. Since
\[ \sum_{n=0}^{\infty} (n+1)^{-1} e^{i n\theta} = P_+ \varphi(\theta), \qquad \text{a.e. } \theta, \]
and $\|\varphi\|_{\infty} = \pi$, it follows that the Hilbert matrix has norm at most $\pi$. As noted above, it also follows that
\[\left| \sum_{n=0}^\infty c_n (n+1)^{-1} \right| \le \pi \| f\|_{1} ,\]
where $f(z)=\sum_{n\geq0} c_n z^n $. In the case of the Hilbert matrix, we have in fact the stronger inequality
\begin{equation}\label{HL} \sum_{n=0}^{\infty} |c_n| (n+1)^{-1}\le \pi \| f\|_{1}, \end{equation}
which was proved by Hardy and Littlewood \cite{HL1}.

We turn next to what is known about multiplicative Hankel forms. Every sequence
$\varrho = (\varrho_1,\,\varrho_2,$ $\,\varrho_3,\,\ldots\,)$ in $\ell^2$ defines in an obvious way a multiplicative Hankel matrix, and we associate with it the corresponding multiplicative Hankel form given by \begin{equation} \label{eq:multhankelform}
	\varrho(a,b) = \sum_{m,n=1}^\infty \varrho_{mn} a_m b_n,
\end{equation}
which initially is defined at least for finitely supported sequences $a$ and $b$ in $\ell^2$. We will now explain, using the Bohr lift,
that every multiplicative Hankel matrix can be uniquely associated with either a Hankel form on $H^2(\T^\infty)\times H^2(\T^\infty)$
or equivalently a (small) Hankel operator acting on $H^2(\T^\infty)$.

If $f$, $g$, and $\varphi$ are Dirichlet series in $\mathcal{H}^2$ with coefficients $a_n$, $b_n$, and $\overline{\varrho_n}$,
respectively, a computation shows that
\[\langle fg, \varphi \rangle_{\mathcal{H}^2} = \varrho(a,b).\]
A formal computation gives that
\[
\langle \mathcal{B}f\mathcal{B}g,\mathcal{B}\varphi\rangle_{L^2(\mathbb{T}^\infty)} = \langle fg,\varphi\rangle_{\mathcal{H}^2},
\]
allowing us to compute the multiplicative Hankel form \eqref{eq:multhankelform} on $\mathbb{T}^\infty$. This means that we
may equivalently study Hankel forms
\begin{equation} \label{eq:polyform}
	H_\Phi(FG) = \langle F G,\Phi\rangle_{L^2(\mathbb{T}^\infty)}, \qquad F, G \in H^2(\T^\infty).
\end{equation}
In our previous considerations we required that $\Phi$ be in $H^2(\T^\infty)$, but there is nothing to prevent us from
considering arbitrary symbols $\Phi$ from $L^2(\mathbb{T}^\infty)$. Hence, each $\Phi$ in $ L^2(\mathbb{T}^\infty)$
induces by \eqref{eq:polyform} a (possibly unbounded) Hankel form $H_\varphi$ on $H^2(\T^\infty) \times H^2(\T^\infty)$.
Of course, this is not a real generalization. Each form $H_\Phi$ is also induced by a symbol $\Psi $ in $H^2(\T^\infty)$; setting
$\Psi = P_+ \Phi$ we have $H_\Phi = H_\Psi$.

On the polydisc, the Hankel form $H_\Phi$ is naturally realized as a (small) Hankel operator $\mathbf{H}_\Phi$, which when
bounded acts as an operator from $H^2(\T^\infty)$ to the anti-analytic space $\overline{H^2}(\mathbb{\T}^\infty)$.
Letting $\overline{P_+}$ denote the orthogonal projection of $L^2(\mathbb{T}^\infty)$ onto $\overline{H^2}(\mathbb{\T}^\infty)$,
we have at least for polynomials $F$ in $H^2(\T^\infty)$ that
\[\mathbf{H}_\Phi F = \overline{P_+}( \overline{\Phi} F).\]

We now come to the question of to which extent Nehari's theorem remains valid in the multiplicative setting. Note first that if $\Psi$
is in $L^\infty(\mathbb{T}^\infty)$, then the corresponding multiplicative Hankel form is bounded, since
\[|H_\Psi(fg)| = |\langle fg, \Psi \rangle| \leq \|f\|_2 \, \|g\|_2\|\,\Psi\|_\infty.\]
We say that $H_\Phi$ has a bounded symbol if there exists $\Psi \in L^\infty(\mathbb{T}^\infty)$ such that
$H_\Phi = H_\Psi$. In \cite{H3}, Helson proved that every Hankel form in the Hilbert--Schmidt class $S_2$ has a bounded symbol,
but it was shown in \cite{OS} that there exist bounded multiplicative Hankel forms without bounded symbols, in sharp contrast
to the classical situation. Hence, there are in fact bounded Hankel forms $H_\Phi$ for which $f \mapsto H_\Phi(f)$ does not define
a bounded functional on $H^1(\mathbb{T}^\infty)$. For when this functional is bounded on $H^1(\mathbb{T}^\infty)$ it has, by
Hahn-Banach, a bounded extension to $L^1(\mathbb{T}^\infty)$ and therefore is given by an $L^\infty(\T^\infty)$-function
$\Psi$ which must satisfy $H_\Phi = H_\Psi$. The result of \cite{OS} was strengthened in \cite{BP}, where it was shown that
there are Hankel forms in Schatten classes $S_p$ without bounded symbols whenever $p>(1-\log \pi/\log 4)^{-1}=5.7388...$

In the opposite direction, we have the following positive result about Hankel forms with bounded symbols, reflecting that when
$\alpha(n)$ is a multiplicative function, variables separate in a natural way so that the classical Nehari theorem applies to each
of the infinitely many copies of the unit circle $\T$.
\begin{theorem} \label{thm6}
Suppose that $\varphi(s):= \sum_{n\geq1} \alpha(n) n^{-s}$ is in $\Ht$ and that $\alpha(n)$ is a multiplicative function.
If $H_{\mathcal{B} \varphi}$ is a bounded Hankel form on $H^2(\T^\infty)\times H^2(\T^\infty)$, then there exist
$\Psi \in L^\infty(\mathbb{T}^\infty)$ such that $\mathcal{B} \varphi=P_+ \Psi$. Moreover, if the function $\alpha(n)$ is
completely multiplicative, then the Hankel form $H_{\mathcal{B} \varphi}$ is always bounded on $H^2(\T^\infty)\times H^2(\T^\infty)$.
\end{theorem}

\begin{proof}  We begin by proving the first statement. To this end, by the assumption that $\alpha(n)$ is a multiplicative function,
we may factor the symbol $\varphi(s)=\sum_{n\geq1} \alpha(n) n^{-s}$ into an Euler product,
	\[\varphi(s) = \sum_{n=1}^\infty \alpha(n) n^{-s} =
\prod_{j=1}^\infty\left(1+\sum_{k=1}^\infty \alpha\big(p_j^k\big)p_j^{-ks}\right) =: \prod_{j=1}^\infty \varphi_j(s),\]
which is absolutely convergent when $\Real s >1/2$. We observe that $\Phi_j := \mathcal{B}\varphi_j$ depends only on $z_j$,
so that $\Phi(z):=\mathcal{B}\varphi(z) = \prod_{j\geq1} \Phi_j(z_j)$. Now a version of Lemma~2 in \cite{BP} can be used
to show that
	\[\|H_\Phi\| = \prod_{j=1}^\infty \|H_{\Phi_j}\|.\]
Since $H_{\Phi_j}$ is a one variable Hankel form, we may appeal to the classical Nehari theorem \cite{N} to infer that there
is some $\Psi_j$ in $L^\infty(\mathbb{T})$ so that $H_{\Phi_j} = H_{\Psi_j}$ and moreover that $\|H_{\Phi_j}\|=\|\Psi_j\|_\infty$.
Setting $\Psi(z): = \prod_{j\geq1} \Psi_j(z_j)$, we conclude that $\|H_\Phi\| = \|\Psi\|_\infty$ and that $\Phi = P_+ \Psi$.

The second statement of	 Theorem~\ref{thm6} is just a reformulation of the fact that the set of bounded point evaluations for
$H^1(\T^\infty)$ is $\D^\infty\cap \ell^2$ \cite{CG}. Following \cite[p. 122]{CG} or the proof of the first part of the present theorem,
we may find corresponding bounded functions explicitly: For every point $z=(z_j)$ on $\T^\infty$, we set
\[\Psi(z)= \prod_{j=1}^{\infty}\frac{1}{1-|\alpha(p_j)|^2} \frac{1-\overline{\alpha(p_j)z_j}}{1-\alpha(p_j)z_j}.\]
	This is a bounded function on $\T^\infty$ because $(\alpha(p_j))_{j\geq1} \in \mathbb{D}^\infty \cap \ell^2$. One may check
that $\B^{-1} P_+ \Psi (s)$ $=\sum_{n\geq1} \alpha(n) n^{-s}$ by a direct computation or by checking that $\Phi$ represents the
functional of point evaluation at $(\alpha(p_j))_{j\geq1}$.
\end{proof}

Because of the factor $1/\log n$, the analytic symbol \eqref{symbol} of the multiplicative Hilbert matrix does not have
multiplicative coefficients, and we know from Theorem~\ref{basicbound} that it is not compact. This means that the preceding
discussion gives no answer to the following question.
\begin{question}
	Does the multiplicative Hilbert matrix have a bounded symbol?
\end{question}
Equivalently, we may ask whether we have
\begin{equation}\label{hb}
	\left| a_1 + \sum_{n=2}^{\infty} \frac{a_n}{\sqrt{n}\log n}\right|\lesssim \| f\|_{\Ho}
\end{equation}
when $f(s)=\sum_{n\geq1} a_n n^{-s}$ is in $\Ho$. We could even ask if the analogue of the Hardy--Littlewood inequality \eqref{HL}
is valid: Does \eqref{hb} hold when we put absolute values on $a_n$, or, in other words, do we have
\[|a_1|+\sum_{n=2}^\infty \frac{|a_n|}{\sqrt{n} \log n} \lesssim \left\|\sum_{n=1}^\infty a_n n^{-s} \right\|_{\Ho}? \]
To see that we could not hope for a better inequality with $\sqrt{n}\log n $ replaced by a function of slower growth, we look at the function
\[ f_N(s):=\left(\sum_{n=1}^N n^{-1/2-s}\right)^2, \]
which has $\| f_N\|_{\Ho} \sim \log N$. On the other hand, we observe that in this case,
\[ \sum_{n=2}^\infty  \frac{|a_n|}{\sqrt{n} \log n} \ge \sum_{n=2}^N \frac{d(n)}{n \log n} \sim \log N, \]
where $d(n)$ is the divisor function and the latter estimate follows by Abel's summation formula.


We observe that the left-hand side of \eqref{hb} can be written as an integral, so that another reformulation of the question is to ask
if the linear functional defined by
\begin{equation} \label{eq:functional}
Lf=\int_{1/2}^\infty f(w) dw
\end{equation}
extends to a bounded linear functional on $\Ho_0$. One of the most important open problems in the theory of Hardy spaces of
Dirichlet series is to determine whether
\begin{equation} \label{embedding}
	\int_{0}^1 |P(1/2+i t)|dt \lesssim \| P\|_{\Ho}
\end{equation}
holds for all Dirichlet polynomials.
If this were the case, then a Carleson measure argument (see \cite[Theorem~4]{OlS}) shows that then we also have
\[ \int_{1/2}^{3/2} |f(w)| dw \lesssim \|f\|_{\Ho} \]
for all $f$ in $\Ho$. The contribution from $\Real(s)\geq3/2$ can be handled with a point estimate. The easiest way (see also \cite{CG}) to deduce a sharp point estimate for $\mathcal{H}^1_0$ is through Helson's inequality \cite{H3}, which states that $\sum_{n\geq1}|a_n|^2/d(n) \leq \|f\|_{\mathcal{H}^1}^2$. For $f \in \mathcal{H}^1_0$ and $\Real(s)=\sigma>1/2$
we get that
\[
|f(s)| \leq \left(\sum_{n=2}^\infty \frac{|a_n|^2}{d(n)}\right)^\frac{1}{2}
\left(\sum_{n=2}^\infty d(n) n^{-2\sigma}\right)^\frac{1}{2} \leq \|f\|_{\mathcal{H}^1_0}
\left(\zeta(2\sigma)^2-1\right)^\frac{1}{2}.
\]
For instance, if $w\geq3/2$ then $|f(w)| \lesssim \|f\|_{\mathcal{H}^1_0} 4^{-w}$. Therefore the validity of the embedding
\eqref{embedding} in fact implies that
\[ \int_{1/2}^\infty |f(w)| dw \lesssim \|f\|_{\Ho_0}.\] 

This inequality is stronger than asking the functional of \eqref{eq:functional} to be bounded on $\Ho_0$, and hence we have
shown that \eqref{embedding} would imply that the multiplicative Hilbert matrix has a bounded symbol. Whether \eqref{embedding}
holds is an open problem that has remained unsolved for many years; we refer to \cite{SS} for a discussion of it.

\bibliographystyle{amsplain}
\bibliography{hilbert}

\end{document}